\documentclass[12pt]{amsart}
\usepackage{amsfonts,amssymb,amscd,amsmath,enumerate,verbatim}
\usepackage[noxcolor]{pstricks}
\usepackage[dvips]{graphicx}
\def\NZQ{\mathbb}               

\def\QQ{{\NZQ Q}}
\def\ZZ{{\NZQ Z}}
\def\RR{{\NZQ R}}

\def\PP{{\NZQ P}}

\def\frk{\mathfrak}               

\def\Phi{{\frk N}}
\def\eb{{\bold e}}

\def\a{\alpha}

\def\opn#1#2{\def#1{\operatorname{#2}}} 
\opn\chara{char} 
\opn\length{\ell} 
\opn\pd{pd} 
\opn\rk{rk}
\opn\projdim{proj\,dim} 
\opn\injdim{inj\,dim} 
\opn\rank{rank}
\opn\depth{depth} 
\opn\grade{grade} 
\opn\height{height}
\opn\embdim{emb\,dim} 
\opn\codim{codim}

\opn\Tr{Tr} 
\opn\bigrank{big\,rank}
\opn\superheight{superheight}
\opn\lcm{lcm}
\opn\trdeg{tr\,deg}
\opn\reg{reg} 
\opn\lreg{lreg} 
\opn\ini{in} 
\opn\lpd{lpd}
\opn\size{size}
\opn\mult{mult}
\opn\dist{dist}
\opn\cone{cone}
\opn\lex{lex}
\opn\rev{rev}
\opn\div{div} \opn\Div{Div} \opn\cl{cl} \opn\Cl{Cl}
\opn\Spec{Spec} \opn\Supp{Supp} \opn\supp{supp} \opn\Sing{Sing}
\opn\Ass{Ass} \opn\Min{Min}
\opn\Ann{Ann} \opn\Rad{Rad} \opn\Soc{Soc}
\opn\Syz{Syz} \opn\Im{Im} \opn\Ker{Ker} \opn\Coker{Coker}
\opn\Am{Am} \opn\Hom{Hom} \opn\Tor{Tor} \opn\Ext{Ext}
\opn\End{End} \opn\Aut{Aut} \opn\id{id} \opn\ini{in}

\opn\nat{nat}
\opn\pff{pf}
\opn\Pf{Pf} \opn\GL{GL} \opn\SL{SL} \opn\mod{mod} \opn\ord{ord}
\opn\Gin{Gin}
\opn\Hilb{Hilb}\opn\adeg{adeg}\opn\std{std}\opn\ip{infpt}
\opn\Pol{Pol}
\opn\sat{sat}
\opn\Var{Var}
\opn\Gen{Gen}

\opn\aff{aff} \opn\con{conv} \opn\relint{relint} \opn\st{st}
\opn\lk{lk} \opn\cn{cn} \opn\core{core} \opn\vol{vol}
\opn\link{link} \opn\star{star}
\opn\gr{gr}

\def\Hc{{\mathcal H}}

\def\Fc{{\mathcal F}}
\def\Pc{{\mathcal P}}

\def\Qc{{\mathcal Q}}
\def\pot#1#2{#1[\kern-0.28ex[#2]\kern-0.28ex]}

\opn\dirlim{\underrightarrow{\lim}}
\opn\inivlim{\underleftarrow{\lim}}
\def\Implies{\ifmmode\Longrightarrow \else
        \unskip${}\Longrightarrow{}$\ignorespaces\fi}
\def\implies{\ifmmode\Rightarrow \else
        \unskip${}\Rightarrow{}$\ignorespaces\fi}
\def\iff{\ifmmode\Longleftrightarrow \else
        \unskip${}\Longleftrightarrow{}$\ignorespaces\fi}

\let\:=\colon
\newtheorem{Theorem}{Theorem}[section]
\newtheorem{Lemma}[Theorem]{Lemma}
\newtheorem{Corollary}[Theorem]{Corollary}
\newtheorem{Proposition}[Theorem]{Proposition}
\newtheorem{Remark}[Theorem]{Remark}

\newtheorem{Example}[Theorem]{Example}

\newtheorem{Definition}[Theorem]{Definition}

\let\epsilon\varepsilon
\let\phi=\varphi
\let\kappa=\varkappa
\opn\dis{dis}
\opn\height{height}
\opn\dist{dist}
\def\pnt{{\raise0.5mm\hbox{\large\bf.}}}

\opn\Lex{Lex}

\begin{document}
\title{Smooth Fano polytopes arising from finite directed graphs}
\author{Akihiro Higashitani}
\address{Akihiro Higashitani,
Department of Mathematics, Kyoto University, Japan 
Kitashirakawa Oiwake-cho, Sakyo-ku, Kyoto, 606-8502, Japan}
\email{ahigashi@math.kyoto-u.ac.jp}
\begin{abstract}
In this paper, we consider terminal reflexive polytopes arising from finite directed graphs 
and study the problem of deciding which directed graphs yield smooth Fano polytopes. 
We show that any centrally symmetric or pseudo-symmetric smooth Fano polytopes 
can be obtained from directed graphs. 
Moreover, by using directed graphs, we provide new examples of smooth Fano polytopes 
whose corresponding varieties admit K\"ahler--Einstein metrics. 
\end{abstract}
\maketitle

\section*{Introduction}
Let $\Pc \subset \RR^d$ be an integral convex polytope, 
that is, a convex polytope whose vertices have integer coordinates, 
of dimension $d$. We say that $\Pc$ is a {\em Fano polytope} 
if the origin of $\RR^d$ is a unique integer point in the interior of $\Pc$. 
\begin{itemize}
\item 
A Fano polytope is called {\em terminal} 
if every integer point on the boundary is a vertex. 
\item 
A Fano polytope is called {\em reflexive} if its dual polytope is integral. 
Here, the dual polytope of a Fano polytope $\Pc$ is the convex polytope 
consisting of $x \in \RR^d$ such that $\langle x, y \rangle \leq 1$ for all $y \in \Pc$, 
where $\langle x, y \rangle$ is the usual inner product of $\RR^d$. 
When $\Pc$ is reflexive, the corresponding toric Fano variety is {\em Gorenstein}. 
\item 
When $\Pc$ is simplicial, the corresponding toric Fano is {\em $\QQ$-factorial}. 
\item 
A Fano polytope is called {\em smooth} if the vertices of each facet 
form a $\ZZ$-basis of $\ZZ^d$. 
\end{itemize}
In particular, smooth Fano polytopes are always terminal, reflexive and simplicial. 

Fano polytopes have been studied by many people. 
{\O}bro \cite{Oeb} constructed the so-called SFP-algorithm which yields 
the complete classification list of the smooth Fano polytopes of dimension $d$ 
for any given positive integer $d$. Casagrande \cite{Cas} proved that 
the number of vertices of a simplicial reflexive polytope is at most $3d$ when $d$ is even, 
and at most $3d-1$ when $d$ is odd. 
In \cite{NO}, Nill and {\O}bro  classified the simplicial reflexive polytopes 
of dimension $d$ with $3d-1$ vertices. 
Reflexive polytopes of dimension $d$ were classified for $d \leq 4$ by Kreuzer and Skarke \cite{KS98, KS00}. 
The study of the classification of Fano polytopes of dimension three was done 
by Kasprzyk \cite{Kas06, Kas08}. 
The combinatorial conditions for what it implies to be terminal and canonical are explained in Reid \cite{Rei83}. 

In this paper, given a finite directed graph $G$, 
we associate a terminal reflexive polytope $\Pc_G$, 
which has been already defined in \cite{hamiltonian} when $G$ is a tournament graph 
and in \cite{MHNOH} when $G$ is a symmetric directed graph. 
We study the characterization problem of directed graphs 
which yield smooth Fano polytopes (Theorem \ref{main}). 
Moreover, we show that any centrally symmetric or pseudo-symmetric smooth Fano polytope 
can be obtained from a directed graph (Theorem \ref{symmetricsmooth}). 
In addition, as an application of Theorem \ref{main}, 
we provide new examples of smooth Fano polytopes whose corresponding varieties 
admit K\"ahler--Einstein metrics (Example \ref{newex}). 
As we see in many examples in Section 4, smooth Fano polytopes 
arising from directed graphs are helpful to understand 
and useful to consider the combinatorics of smooth Fano polytopes.

\section{Fano polytopes arising from finite directed graphs}

In this section, we construct an integral convex polytope 
associated with a finite directed graph and 
discuss the condition with which the directed graph yields a Fano polytope. 
For most parts of this section, we refer to \cite{HibiH, MHNOH, OhsugiHibi, hamiltonian}. 

Let $G=(V(G),A(G))$ be a finite directed graph 
on the vertex set $V(G)=\{1,\ldots,d\}$ with the arrow set $A(G)$. 
Here an {\em arrow} of $G$ is an ordered pair of two vertices $(i,j)$, 
where $1 \leq i \not= j \leq d$, and the {\em arrow set} $A(G)$ of $G$ is the set of all the arrows of $G$. 
In particular, we regard that $(i,j)$ and $(j,i)$ are distinct arrows. 
We also define an undirected graph $\widetilde{G}$ from a directed graph $G$ as follows: 
$\widetilde{G}$ consists of the vertex set $V(G)$ 
and the edge set $E(G)=\{\{i,j\} \in V(G) \times V(G) : (i,j) \text{ or } (j,i) \in A(G)\}$. 
We call a pair of two vertices without ordering $\{i,j\} \in E(G)$ an {\em edge} of $G$. 

Throughout this paper, we allow that both $(i,j)$ and $(j,i)$ are simultaneously contained in $A(G)$ 
and $\widetilde{G}$ is connected. 


\begin{Definition}{\em 
Let $\eb_1,\ldots,\eb_d$ be the standard basis of $\RR^d$. 
For an arrow $\vec{e}=(i,j)$ of $G$, we define $\rho(\vec{e}) \in \RR^d$ by setting 
$\rho(\vec{e})=\eb_i-\eb_j$. Moreover, we write $\Pc_G \subset \RR^d$ 
for the convex hull of $\{\rho(\vec{e}) : \vec{e} \in A(G)\}$. 
}\end{Definition}

\begin{Remark}{\em 
In \cite{hamiltonian}, $\Pc_G$ is introduced for a tournament graph $G$, 
which is called the {\em edge polytope} of $G$, and some properties on $\Pc_G$ 
are studied in \cite[Section 1]{hamiltonian}. Similarly, in \cite[Section 4]{MHNOH}, 
$\Pc_G$ is defined for a symmetric graph $G$, which is denoted by $\Pc_G^{\pm}$, 
and called the {\em symmetric edge polytope} of $G$. 
}\end{Remark}

Let $\Hc \subset \RR^d$ denote the hyperplane defined by the equation $x_1+\cdots+x_d=0$. 
Since each integer point of $\{\rho(\vec{e}) : \vec{e} \in A(G)\}$ lies on $\Hc$, 
one has $\Pc_G \subset \Hc$. Thus, $\dim(\Pc_G) \leq d-1$. 
First, we discuss the dimension of $\Pc_G$. 
A sequence $\Gamma=(i_1,\ldots,i_l)$ of vertices of $G$ is called a {\em cycle} 
if $i_j \not= i_{j'}$ for $1 \leq j < j' \leq l$ and 
either $(i_j, i_{j+1})$ or $(i_{j+1},i_j)$ is an arrow of $G$ for each $1 \leq j \leq l$, where $i_{l+1}=i_1$. 
In other words, the edges $\{i_1,i_2\},\{i_2,i_3\},\ldots,\{i_l,i_1\}$ form a cycle in $\widetilde{G}$. 
For short, we often write $\Gamma=(\vec{e_1},\ldots,\vec{e_l})$, 
where $\vec{e_j} = (i_j, i_{j+1})$ or $\vec{e_j} = (i_{j+1}, i_j)$ for $1 \leq j \leq l$. 
The length of a cycle is the number of vertices (or edges) forming a cycle. 
For a cycle $\Gamma=(\vec{e_1},\ldots,\vec{e_l})$ in $G$, 
let $\Delta_{\Gamma}^{(+)}=\{ \vec{e_j} \in \{\vec{e_1},\ldots,\vec{e_l}\} : 
\vec{e_j}=(i_j, i_{j+1})\}$ and $\Delta_{\Gamma}^{(-)}=
\{\vec{e_1},\ldots,\vec{e_l}\} \setminus \Delta_{\Gamma}^{(+)}$. 
A cycle $\Gamma$ is called {\em nonhomogeneous} 
if $|\Delta_{\Gamma}^{(+)}| \not= |\Delta_{\Gamma}^{(-)}|$ 
and {\em homogeneous} if $|\Delta_{\Gamma}^{(+)}| = |\Delta_{\Gamma}^{(-)}|$, 
where $|X|$ denotes the cardinality of a finite set $X$. 
We note that two arrows $(i,j)$ and $(j,i)$ form a nonhomogeneous cycle of length two, 
although these do not form a cycle in $\widetilde{G}$. 
We also note that every odd cycle is nonhomogeneous. 
(Here odd (resp. even) cycle is a cycle of odd (resp. even) length.) 
The following result can be proved similarly to \cite[Proposition 1.3]{OhsugiHibi} and \cite[Lemma 1.1]{hamiltonian}.

\begin{Proposition}[{\cite[Proposition 1.3]{OhsugiHibi} and \cite[Lemma 1.1]{hamiltonian}}]\label{dim}
One has $\dim(\Pc_G) = d-1$ if and only if $G$ contains a nonhomogeneous cycle. 
\end{Proposition}

We assume that $G$ has at least one nonhomogeneous cycle. 

Next, we investigate directed graphs which define Fano polytopes. 
Once we know that $\Pc_G$ is a Fano polytope, 
one can verify that it is terminal and reflexive (\cite[Lemma 1.4 and Lemma 1.5]{HibiH}). 
The following result can be proved similarly to \cite[Proposition 4.2]{MHNOH} and \cite[Lemma 1.2]{hamiltonian}.

\begin{Proposition}[{\cite[Proposition 4.2]{MHNOH} and \cite[Lemma 1.2]{hamiltonian}}]\label{tGF}
An integral convex polytope $\Pc_G \subset \Hc$ is a terminal reflexive polytope of dimension $d-1$ 
if and only if every arrow of $G$ appears in a directed cycle in $G$, 
where a cycle $\Gamma$ is called a directed cycle if either $\Delta_{\Gamma}^{(+)}$ or $\Delta_{\Gamma}^{(-)}$ is empty. 
\end{Proposition}

Hereafter, we assume that every arrow of $G$ appears in a directed cycle in $G$. 
Notice that by this condition, 
$G$ has a nonhomogeneous cycle since every directed cycle is nonhomogeneous. 

\begin{Example}{\em 
Let $G$ be a directed graph on the vertex set $\{1,2,3\}$ 
with the arrow set $\{(1,2),(2,1),(2,3),(3,1)\}$. Then 
$G$, $\rho(\vec{e})$'s and $\Pc_G$ are as Figure \ref{example1}: 

\begin{figure}[htb!]
\centering
\includegraphics[scale=0.3]{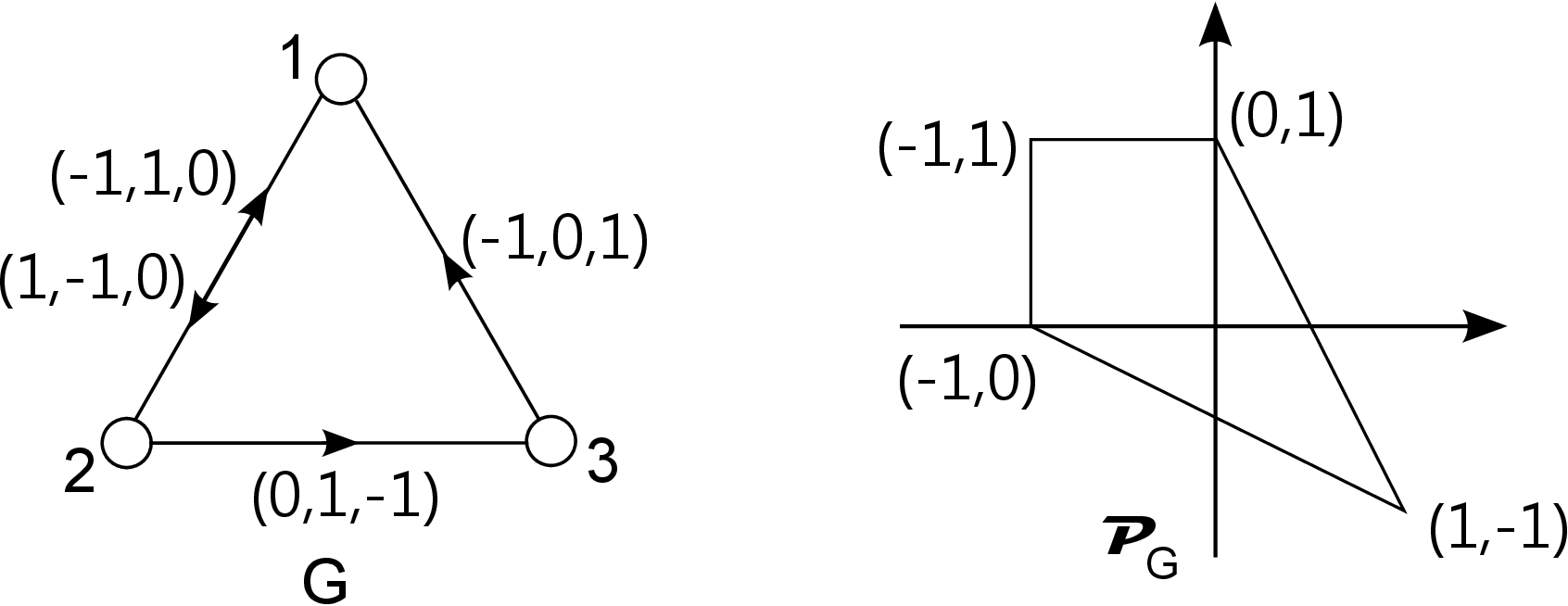}
\caption{}\label{example1}
\end{figure}

\noindent
Remark that the arrows $(1,2),(2,3),(3,1)$ (resp. the arrows $(1,2),(2,1)$) form 
a directed cycle of length three (resp. length two). 
In the picture of $\Pc_G$, we ignore the third coordinate of each integer point. 
Then the convex polytope $\Pc_G$ of this example becomes 
a smooth (in particular, terminal and reflexive) Fano polytope of dimension two. 
}\end{Example}

\begin{Remark}{\em 
In \cite{HibiH}, terminal reflexive polytopes $\Qc_P$ arising from 
finite partially ordered sets $P$ are introduced. 
Let $P=\{y_1,\ldots,y_d\}$ be a partially ordered set 
and $\widehat{P}=P \cup \{y_0,y_{d+1}\}$, where $y_0=\hat{0}$ and $y_{d+1}=\hat{1}$. 
Then we can regard $\hat{P}$ as a directed graph on the vertex set $\{0,1,\ldots,d+1\}$ 
with the arrow set $$\{(i,j) : y_j \text{ covers } y_i \}.$$ 
Identifying $0$ with $d+1$ as the same vertex, we obtain a directed graph $G_P$ 
on the vertex set $\{1,\ldots,d+1\}$. 
Then $\Qc_P$ is nothing but $\Pc_{G_P}$. 
Therefore, terminal reflexive polytopes associated with directed graphs 
are a natural generalization of those defined in \cite{HibiH}. 
We can study these polytopes in Section 2 in a similar way. 
}\end{Remark}

\section{When is $\Pc_G$ smooth ?}

In this section, we consider the problem of which directed graphs yield smooth Fano polytopes. 
First, we prove the following 
\begin{Lemma}\label{lemma}
{\em (a)} 
Let $C=(\vec{e_1},\ldots,\vec{e_l})$ be a cycle in $G$. 
If there exists a facet $\Fc$ of $\Pc_G$ with 
$\{\rho(\vec{e_1}),\ldots,\rho(\vec{e_l})\} \subset \Fc$, then $C$ is homogeneous. \\
{\em (b)} Suppose $(i,j) \in A(G)$ and $(j,i) \in A(G)$. 
If $\rho((i,j))$ is contained in some facet $\Fc$ of $\Pc_G$, then $\rho((j,i))$ does not belong to $\Fc$. 
\end{Lemma}
\begin{proof}
(a) Let $a_1x_1+\cdots+a_dx_d=1$, where each $a_i \in \QQ$, 
denote the equation of the supporting hyperplane of $\Pc_G$ which defines a facet $\Fc$. 
Let $\vec{e_j} \in \{(i_j,i_{j+1}),(i_{j+1},i_j)\}$ for $1 \leq j \leq l$, where $i_{l+1}=i_1$. 
It then follows that $$\sum_{j=1}^l(a_{i_j}-a_{i_{j+1}}) 
=\sum_{\vec{e_j} \in \Delta_C^{(+)}}(a_{i_j}-a_{i_{j+1}}) 
-\sum_{\vec{e_j} \in \Delta_C^{(-)}}(a_{i_{j+1}}-a_{i_j})
=|\Delta_C^{(+)}|-|\Delta_C^{(-)}|=0.$$ 
Hence, $C$ must be homogeneous. 

(b) Similarly, we set $a_1x_1+\cdots+a_dx_d=1$ as above and suppose that 
$\rho((i,j))$ lies on this supporting hyperplane. Then one has $a_i-a_j=1$. Thus, $a_j-a_i=-1$. 
This implies that $\rho((j,i))$ cannot be contained in the same supporting hyperplane. 
\end{proof}

Next, we define two pieces of notation, $\mu_C$ and $\dist_G$. 
Let $C = (\vec{e_1},\ldots,\vec{e_l})$ be a homogeneous cycle in $G$ of length $l$, 
where $\vec{e_j}$ is either $(i_j,i_{j+1})$ or $(i_{j+1},i_j)$ for $1 \leq j \leq l$ with $i_{l+1}=i_1$. 
Then there exists a unique function 
$$\mu_C:\{i_1,\ldots,i_l\} \rightarrow \ZZ_{\geq 0}$$
such that 
\begin{itemize}
\item $\mu_C(i_{j+1})=\mu_C(i_j)-1$ 
(resp. $\mu_C(i_{j+1})=\mu_C(i_j)+1$) 
if $\vec{e_j}=(i_j,i_{j+1})$ (resp. $\vec{e_j}=(i_{j+1},i_j)$) for $1 \leq j \leq l$; 
\item $\min(\{\mu_C(i_1),\ldots,\mu_C(i_l)\})=0$. 
\end{itemize}

For two distinct vertices $i$ and $j$ of $G$, 
the {\em distance} from $i$ to $j$, denoted by $\dist_G(i,j)$, 
is the length of the shortest directed path in $G$ from $i$ to $j$. 
If there exists no directed path from $i$ to $j$, 
then the distance from $i$ to $j$ is defined to be infinity. 

\smallskip

\begin{Theorem}\label{main}
Let $G$ be a connected directed graph on the vertex set $\{1,\ldots,d\}$ 
satisfying that every arrow of $G$ appears in a directed cycle in $G$. 
Then the following conditions are equivalent: \\
{\em (i)} $\Pc_G$ is simplicial; \;\;\;\;\; {\em (ii)} $\Pc_G$ is smooth; \\
{\em (iii)} $G$ possesses no homogeneous cycle $C=(\vec{e_1},\ldots,\vec{e_l})$ such that 
\begin{eqnarray}\label{eq1}
\mu_C(i_a)-\mu_C(i_b) \leq \dist_G(i_a,i_b)
\end{eqnarray}
for all $1 \leq a,b \leq l$, where $\vec{e_j}$ is $(i_j,i_{j+1})$ or $(i_{j+1},i_j)$ 
for $1 \leq j \leq l$ with $i_{l+1}=i_1$. 
\end{Theorem}
\begin{proof}
{\bf ((i) $\Rightarrow$ (iii))} 
Suppose that $G$ possesses a homogeneous cycle $C$ satisfying \eqref{eq1} 
and let $C=(\vec{e_1},\ldots,\vec{e_l})$ be such cycle, 
where $\vec{e_j}$ is either $(i_j,i_{j+1})$ or $(i_{j+1},i_j)$ for $1 \leq j \leq l$ with $i_{j+1}=i_1$. 
Then one has $\sum_{j=1}^lq_j\rho(\vec{e_j})=(0,\ldots,0),$ 
where $q_j=1$ (resp. $q_j=-1$) if $\vec{e_j}=(i_j,i_{j+1})$ 
(resp. if $\vec{e_j}=(i_{j+1},i_j)$) for $1 \leq j \leq l$. 
Since $C$ is homogeneous, one has $\sum_{j=1}^lq_j=0$, which implies that 
the integer points $\rho(\vec{e_1}),\ldots,\rho(\vec{e_l})$ are not affinely independent. 

Let $v_j=\rho(\vec{e_j})$ for $1 \leq j \leq l$. 
In order to show that $\Pc_G$ is not simplicial, it suffices to find a face of $\Pc_G$ containing $v_1,\ldots,v_l$. 
Let $a_1,\ldots,a_d$ be integers. We write $\Hc \subset \RR^d$ for the hyperplane 
defined by the equation $a_1x_1+\cdots+a_dx_d=1$ and 
$\Hc^{(+)} \subset \RR^d$ for the closed half space 
defined by the inequality $a_1x_1+\cdots+a_dx_d \leq 1$. 
We will show that for suitable $a_1,\ldots,a_d$, we make $\Hc$ a supporting hyperplane of a face $\Fc$ of $\Pc_G$ 
satisfying $\{v_1,\ldots,v_l\} \subset \Fc$ and $\Pc_G \subset \Hc^{(+)}$. 

First, let $a_{i_j}=\mu_C(i_j)$ for $1 \leq j \leq l$. 
It then follows easily that $v_j$ lies on the hyperplane 
defined by the equation $\sum_{j=1}^la_{i_j}x_{i_j}=1$. 

Next, we determine $a_k$ with $k \in A=\{1,\ldots,d\} \setminus \{i_1,\ldots,i_l\}$. 
We set $$a_k=\max(\{a_{i_j}-\dist_G(i_j,k)\} \cup \{0\}).$$ 
In particular, we have $a_k=0$ when there is no $i_j$ with $\dist_G(i_j,k) < \infty$. 
Here, we notice that one has 
\begin{eqnarray}\label{eq2}
a_k \leq a_k', 
\end{eqnarray}
where $a_k'=\min(\{a_{i_{j'}}+\dist_G(k,i_{j'})\})$. 
In fact, if $a_k > a_k'$, then there are $i_j$ and $i_{j'}$ such that 
$\dist_G(i_j,k) < \infty, \dist_G(k,i_{j'}) < \infty$ and 
$a_{i_j} - \dist_G(i_j,k) > a_{i_{j'}} + \dist_G(k,i_{j'})$. 
Since $\dist_G(i_j,k) + \dist_G(k,i_{j'}) \geq \dist_G(i_j,i_{j'})$, one has 
$$\mu_C(i_j)-\mu_C(i_{j'})=a_{i_j}-a_{i_{j'}} 
> \dist_G(i_j,k) + \dist_G(k,i_{j'}) \geq \dist_G(i_j,i_{j'}).$$ 
This contradicts \eqref{eq1}. 

Now we finish determining the integers $a_1,\ldots,a_d$. Since each $v_j$ lies on $\Hc$, 
in order to show that $\Fc$ is defined by $\Hc$, it suffices to show $\Pc_G \subset \Hc^{(+)}.$ 

Let $(i,j) \in A(G)$. When $i \in \{i_1,\ldots,i_l\}$ and $j \in A$, 
then one has $a_j \geq \max(\{a_i-1,0\})$ by the definition of $a_j$. Hence, $a_i - a_j \leq 1$. 
If $i \in A$ and $j \in \{i_1,\ldots,i_l\}$, then one has $a_i \leq a_j + 1$ by \eqref{eq2}. Hence, $a_i - a_j \leq 1$. Let 
\begin{align*}
&B=\{k \in A : \text{there is } i_j \text{ with } \dist_G(i_j,k) < \infty\} \;\;\text{ and } \;\;\\
&C=\{k \in A : \text{there is } i_{j'} \text{ with } \dist_G(k,i_{j'}) < \infty\}.
\end{align*}
Again, let $(i,j) \in A(G)$. In each of the nine cases below, 
by a routine computation, we can easily show that $\rho((i,j))$ is in $\Hc^{(+)}$. 
\begin{align*}
&(1) \; i \in B \setminus C \;\text{ and }\; j \in B \setminus C; \;\;\; 
&(2) \; i \in C \setminus B \;\text{ and }\; j \in C \setminus B; \\
&(3) \; i \in C \setminus B \;\text{ and }\; j \in B \setminus C; \;\;\; 
&(4) \; i \in C \setminus B \;\text{ and }\; j \in B \cap C; \\
&(5) \; i \in C \setminus B \;\text{ and }\; j \not\in B \cup C; \;\;\;
&(6) \; i \in B \cap C \;\text{ and }\; j \in B \setminus C; \\
&(7) \; i \in B \cap C \;\text{ and }\; j \in B \cap C; \;\;\;
&(8) \; i \not\in B \cup C \;\text{ and }\; j \in B \setminus C; \\
&(9) \; i \not\in B \cup C \;\text{ and }\; j \not\in B \cup C. 
\end{align*}

For example, a routine computation of (1) is as follows. 
When $a_i=0$, since $a_j \geq 0$, one has $a_i-a_j \leq 0 \leq 1$. 
When $a_i >0$, since $a_j \geq a_i -1$, one has $a_i-a_j \leq 1$. 

Therefore, it follows that $\Hc$ is 
a supporting hyperplane of a face of $\Pc_G$ which is not a simplex. 

{\bf ((iii) $\Rightarrow$ (i))} 
Suppose that $\Pc_G$ is not simplicial, i.e., $\Pc_G$ contains a facet $\Fc$ which is not a simplex. 
Let $v_1,\ldots,v_n$ be the vertices of $\Fc$, where $n > d-1$, 
and $\vec{e_1},\ldots,\vec{e_n}$ the arrows with $v_j = \rho(\vec{e_j})$ for $1 \leq j \leq n$. 
We write $\Hc \subset \RR^d$ for the supporting hyperplane $a_1x_1+\cdots+a_dx_d=1$ defining $\Fc$. 
Since $v_1,\ldots,v_n$ are not affinely independent, 
there is $(r_1,\ldots,r_n) \in \RR^n$ with $(r_1,\ldots,r_n) \not= (0,\ldots,0)$ 
satisfying $\sum_{j=1}^nr_j=0$ and $\sum_{j=1}^nr_jv_j=(0,\ldots,0)$. 
By removing $r_j$ with $r_j = 0$, we may assume that $\sum_{j=1}^{n'}r_jv_j=(0,\ldots,0)$, 
where $r_j \not= 0$ for $1 \leq j \leq n'$ with $\sum_{j=1}^{n'}r_j=0$. 
Let $\vec{e_j}=(i_j,i_j')$ with $1 \leq i_j,i_j' \leq d$ and 
let $G'$ denote the subgraph of $G$ with the arrow set $\{\vec{e_1},\ldots,\vec{e_{n'}}\}$. 
If $\deg_{G'}(i_j)=1$ or $\deg_{G'}(i_j')=1$, then $r_j=0$, a contradiction. 
(For a graph $H$ and its vertex $v$, $\deg_H(v)$ denotes the number of arrows $\vec{e}$ in $H$ 
such that $\vec{e}$ looks like $(v,v')$ or $(v',v)$.) 
Thus, $\deg_{G'}(i_j) \geq 2$ and $\deg_{G'}(i_j') \geq 2$. 
By Lemma \ref{lemma} (b), since $\{\rho(\vec{e_1}),\ldots,\rho(\vec{e_{n'}})\} \subset \Fc$, 
it cannot happen that $(i_j,i_j')=(i_k',i_k)$ for some $1 \leq j \not= k \leq n'$. 
Moreover, since every vertex in $G'$ is at least degree two, 
$G'$ is not a tree. Hence $G'$ contains a cycle, which should be homogeneous by Lemma \ref{lemma} (a). 

Let $C=(\vec{e_1},\ldots,\vec{e_l})$ be a homogeneous cycle in $G$, 
where $\vec{e_j}$ is either $(i_j,i_{j+1})$ or $(i_{j+1},i_j)$ for $1 \leq j \leq l$ with $i_{j+1}=i_1$. 
Our goal is to show that $C$ satisfies the inequality \eqref{eq1}. 

Let $\Gamma=(k_0,k_1,\ldots,k_m)$ be a directed shortest path in $G$ 
such that $k_0$ and $k_m$ belong to $\{i_1,\ldots,i_l\}$. 
On the one hand, since $\eb_{k_j}-\eb_{k_{j+1}} \in \Pc_G$, 
one has $a_{k_j}-a_{k_{j+1}} \leq 1$ for $0 \leq j \leq m-1$. 
Hence, $a_{k_0}-a_{k_m} \leq m=\dist_G(k_0,k_m)$. 
On the other hand, we have $a_{k_0}-a_{k_m} = \mu_C(k_0)-\mu_C(k_m).$ 
Thus, $\mu_C(k_0)-\mu_C(k_m) \leq \dist_G(k_0,k_m)$. 
Therefore, the required inequality \eqref{eq1} holds. 

{\bf ((i) $\Rightarrow$ (ii))} 
Suppose that $\Pc_G$ is simplicial. 
Then there are just $(d-1)$ vertices in each facet, which are linearly independent. Let 
$M$ be the $(d-1) \times d$ matrix whose row vectors $v_1,\ldots,v_{d-1} \in \ZZ^d$ are 
the vertices of a facet of $\Pc_G$ 
and $M'$ the $(d-1) \times (d-1)$ submatrix of $M$ 
ignoring the $d$th column of $M$. 
From the theory of totally unimodular matrices \cite{Sch}, 
the determinant of $M'$ is equal to $\pm 1$, 
which means that $\Pc_G$ is smooth. 

{\bf ((ii) $\Rightarrow$ (i))} 
In general, every smooth Fano polytope is simplicial. 
\end{proof}

For a directed graph $G$, we say that $G$ is {\em symmetric} 
if $(j,i)$ belongs to $A(G)$ for every $(i,j) \in A(G)$, that is, $2|E(G)|=|A(G)|$. 
Note that when $G$ is symmetric, every arrow of $G$ is contained 
in a directed cycle of length two, 
so $\Pc_G$ is always a terminal reflexive polytope. 

A connected undirected graph $G$ is called {\em two-connected} 
if the induced subgraph with the vertex set $V(G) \backslash \{i\}$ 
is connected for any $i \in V(G)$. A subgraph is called a {\em two-connected component} of $G$ 
if it is a maximal two-connected subgraph in $G$.

For symmetric directed graphs, we obtain the following 
\begin{Corollary}\label{noeven}
Assume that $G$ is a connected symmetric directed graph. 
Then the following conditions are equivalent: \\
{\em (i)} $\Pc_G$ is simplicial; \;\;\;\;\; {\em (ii)} $\Pc_G$ is smooth; 
\;\; {\em (iii)} $\widetilde{G}$ contains no even cycle; \\
{\em (iv)} every two-connected component of $\widetilde{G}$ is either one edge or an odd cycle. 
\end{Corollary}
\begin{proof}
{\bf ((i) $\Leftrightarrow$ (ii))} 
This equivalence follows from Theorem \ref{main}. 

{\bf ((i) $\Rightarrow$ (iii))} 
Suppose that $\widetilde{G}$ possesses an even cycle $C$ of length $2l$. 
Let $C=(e_{i_1},\ldots,e_{i_{2l}})$ be a cycle, 
where $e_j=\{i_j,i_{j+1}\}$ for $1 \leq j \leq 2l$ with $i_{2l+1}=i_1$. 
Since $G$ is symmetric, there are arrows 
$(i_2,i_1),(i_2,i_3),(i_4,i_3),(i_4,i_5),\ldots,(i_{2l},i_{2l-1}) \text{ and }(i_{2l},i_1)$ 
in $G$. We define $v_1,\ldots,v_{2l} \in \RR^d$ by setting 
\begin{eqnarray*}
v_j=
\begin{cases}
\rho((i_{j+1},i_j)), \;\;\;\;\;\; &j=1,3,\ldots,2l-1, \\
\rho((i_j,i_{j+1})), &j=2,4,\ldots,2l. 
\end{cases}
\end{eqnarray*}
Then one has $$\sum_{j=1}^lv_{2j-1} - \sum_{j=1}^lv_{2j} = (0,\ldots,0).$$ 
Thus, $v_1,\ldots,v_{2l}$ are not affinely independent. 
Hence, we may show that there is a face $\Fc$ of $\Pc_G$ with $\{v_1,\ldots,v_{2l}\} \subset \Fc$. 

Now, we have $v_{2j-1}=-{\bf e}_{i_{2j-1}}+{\bf e}_{i_{2j}}$ and 
$v_{2j}={\bf e}_{i_{2j}}-{\bf e}_{i_{2j+1}}$ for $1 \leq j \leq l$. 
Thus, $v_1,\ldots,v_{2l}$ lie on the hyperplane $\Hc \subset \RR^d$ 
defined by the equation $x_{i_2}+x_{i_4}+\cdots+x_{i_{2l}}=1$. In addition, 
it is clear that $\rho(\vec{e})$ is contained in $\Hc^{(+)} \subset \RR^d$ for any arrow $\vec{e}$ of $G$. 
Hence, $\Hc$ is a supporting hyperplane defining a face $\Fc$ of $\Pc_G$ 
with $\{v_1,\ldots,v_{2l}\} \subset \Fc$. Therefore, $\Pc_G$ is not simplicial.

{\bf ((iii) $\Rightarrow$ (iv))} 
We prove this implication by elementary graph theory. 
Suppose that there is a two-connected component of $\widetilde{G}$ 
which is neither one edge nor an odd cycle. Let $G'$ be such two-connected subgraph of $\widetilde{G}$. 
Now, an arbitrary two-connected graph with at least three vertices can be obtained by the following method: 
starting from a cycle and repeatedly appending an $H$-path to a graph $H$ 
that has been already constructed. (Consult, e.g., \cite{Wilson}.) 
Since $G'$ is not one edge, $G'$ has at least three vertices. 
Thus, there is one cycle $C_1$ and $(m-1)$ paths $\Gamma_2,\ldots,\Gamma_m$ such that 
$G'=C_1 \cup \Gamma_2 \cup \cdots \cup \Gamma_m$. 
Since $G'$ is not an odd cycle, one has $G'=C_1$, where $C_1$ is an even cycle, or $m > 1$. 
Suppose that $m>1$ and $C_1$ is an odd cycle. 
Let $v$ and $w$ be distinct two vertices of $C_1$ which are intersected with $\Gamma_2$. 
Then there are two paths in $C_1$ from $v$ to $w$. 
Since $C_1$ is odd, the parities of the lengths of such two paths are different. 
By attaching the path $\Gamma_2$ to one or another of such two paths, 
we can construct an even cycle. Therefore, there exists an even cycle. 

{\bf ((iv) $\Rightarrow$ (i))} 
Suppose that each two-connected component of $\widetilde{G}$ is 
either one edge or an odd cycle. 
Then there is no homogeneous cycle in $G$. 
Hence, by Theorem \ref{main}, $\Pc_G$ is simplicial. 
\end{proof}

\section{The case where $\widetilde{G}$ possesses no even cycle} 


In this section, we show that every pseudo-symmetric smooth Fano polytope can be obtained from 
some directed graph whose corresponding undirected graph contains no even cycle. 
This includes the case of centrally symmetric smooth Fano polytopes. 


Let $\Pc \subset \RR^d$ be a Fano polytope. 
\begin{itemize} 
\item We call $\Pc$ {\em centrally symmetric} if $\Pc=-\Pc=\{-\alpha : \alpha \in \Pc\}$. 
\item We call $\Pc$ {\em pseudo-symmetric} 
if there is a facet $\Fc$ of $\Pc$ such that $-\Fc$ is also its facet. 
Note that every centrally symmetric polytope is pseudo-symmetric. 
\item A {\em del Pezzo polytope} of dimension $2k$ is a convex polytope 
$$\con(\{\pm\eb_1,\ldots,\pm\eb_{2k},\pm(\eb_1+\cdots+\eb_{2k})\}),$$ 
whose corresponding variety is called a {\em del Pezzo variety} $V^{2k}$. 
Note that del Pezzo polytopes are centrally symmetric smooth Fano polytopes. 
\item A {\em pseudo del Pezzo polytope} of dimension $2k$ is a convex polytope 
$$\con(\{\pm\eb_1,\ldots,\pm\eb_{2k},\eb_1+\cdots+\eb_{2k}\}),$$ 
whose corresponding variety is called a {\em pseudo del Pezzo variety} $\widetilde{V}^{2k}$. 
Note that pseudo del Pezzo polytopes are pseudo-symmetric smooth Fano polytopes. 
\item Let us say that $\Pc$ {\em splits} into $\Pc_1$ and $\Pc_2$ 
if $\Pc$ is the convex hull of two Fano polytopes 
$\Pc_1 \subset \RR^{d_1}$ and $\Pc_2 \subset \RR^{d_2}$ with $d=d_1+d_2$, i.e., 
by renumbering coordinates, we have 
$$\Pc = \con(\{ (\a_1,0), (0,\a_2) \in \RR^d : \a_1 \in \Pc_1,\a_2 \in \Pc_2 \}).$$ 
\end{itemize}

There is a well-known fact on the characterization of 
centrally symmetric or pseudo-symmetric smooth Fano polytopes. 

\begin{Theorem}[{\cite{VK}}]\label{delPezzo} 
Any centrally symmetric smooth Fano polytope splits into 
copies of the closed interval $[-1,1]$ or a del Pezzo polytope. 
\end{Theorem}
\begin{Theorem}[{\cite{Ewald, VK}}]\label{pdelPezzo} 
Any pseudo-symmetric smooth Fano polytope splits into 
copies of the closed interval $[-1,1]$ or a del Pezzo polytope 
or a pseudo del Pezzo polytope. 
\end{Theorem}

We note that Nill \cite{Nill} studies pseudo-symmetric simplicial reflexive polytopes. 


Somewhat surprisingly, we can give the complete characterization of 
centrally symmetric or pseudo-symmetric smooth Fano polytopes 
by means of directed graphs. 
\begin{Theorem}\label{symmetricsmooth}
{\em (a)} Any centrally symmetric smooth Fano polytope is obtained 
from a symmetric directed graph whose corresponding undirected graph has no even cycle. \\
{\em (b)} Any pseudo-symmetric smooth Fano polytope is obtained 
from a directed graph whose corresponding undirected graph has no even cycle. 
\end{Theorem}
\begin{proof}
First, we prove (b). 
Let $\Pc$ be an arbitrary pseudo-symmetric smooth Fano polytope of dimension $d$. 
By Theorem \ref{pdelPezzo}, $\Pc$ splits into $\Pc_1,\ldots,\Pc_m$ 
which are copies of the closed interval $[-1,1]$ or a del Pezzo polytope or a pseudo del Pezzo polytope. 
Let $\Pc_1,\ldots,\Pc_{m'}$ be del Pezzo polytopes, 
$\Pc_{m'+1},\ldots,\Pc_{m''}$ pseudo del Pezzo polytopes and 
$\Pc_{m''+1},\ldots,\Pc_m$ the closed interval $[-1,1]$. 
Then the following easily follow. 
\begin{itemize}
\item Let, say, $\Pc_1$ be a del Pezzo polytope of dimension $2k_1$ and 
$G_1$ a symmetric directed graph with its arrow set 
$$A(G_1)=\{(i,i+1) : 1 \leq i \leq 2k_1\} \cup \{(1,2k_1+1),(2k_1+1,1)\}.$$ 
Then $G_1$ is an odd cycle, i.e., there is no even cycle, 
so $\Pc_{G_1}$ is smooth by Corollary \ref{noeven} and we can check that 
$\Pc_{G_1}$ is unimodularly equivalent to $\Pc_1$. 
\item Let, say, $\Pc_{m'+1}$ be a pseudo del Pezzo polytope of dimension $2k_1$ and 
$G_1'$ a directed graph with its arrow set $$A(G_1')=A(G_1) \setminus \{(2,1)\},$$ 
i.e., we miss one arrow from $G_1$. 
Then we can also check that $\Pc_{G_1'}$ is unimodularly equivalent to $\Pc_{m'+1}$. 
\item A directed graph consisting of only one symmetric edge 
yields the smooth Fano polytope of dimension one, that is, the closed interval $[-1,1]$. 
\end{itemize}
By connecting the above graphs with one vertex, 
we obtain the directed graph whose corresponding undirected graph has no even cycle 
and this yields the required smooth Fano polytope $\Pc$. 

Moreover, del Pezzo polytopes and the closed interval $[-1,1]$ are 
constructed by symmetric directed graphs. Therefore, 
by Theorem \ref{delPezzo}, we can also find the symmetric directed graph $G$ such that $\widetilde{G}$ 
has no even cycle and $\Pc_G$ is unimodularly equivalent to $\Pc$ 
for any centrally symmetric smooth Fano polytope $\Pc$, proving (a). 
\end{proof}

\section{Examples of smooth Fano polytopes $\Pc_G$}

In this section, we provide some interesting examples of smooth Fano polytopes arising from directed graphs. 

\begin{Example}{\em 
Let $G$ be a directed cycle of length $d+1$. Then $\Pc_G$ is a smooth Fano polytope 
whose corresponding toric Fano variety is a $d$-dimensional projective space $\PP^d$. 
The left-hand side (resp. right-hand side)  of the graph in Figure \ref{example_proj} 
yields a smooth Fano polytope which corresponds to $\PP^5$ (resp. $\PP^3 \times \PP^3$). 
Here each two-connected component of a directed graph 
corresponds to each direct factor of the corresponding toric Fano variety.

\begin{figure}[htb!]
\centering
\includegraphics[scale=0.3]{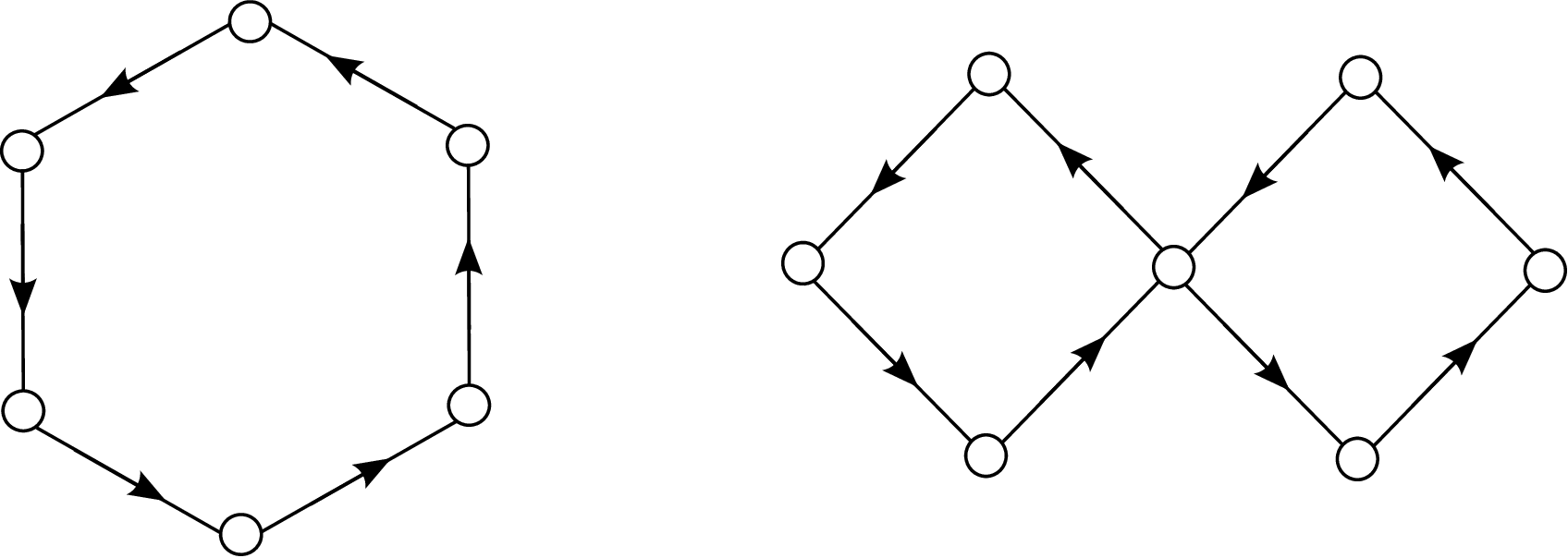}
\caption{directed graphs yielding $\PP^5$ and $\PP^3 \times \PP^3$}\label{example_proj}
\end{figure}
}\end{Example}

\begin{Example}{\em 
(a) When $G$ is a symmetric directed graph without even cycles, $\Pc_G$ is a smooth Fano polytope 
whose corresponding toric Fano variety is a direct product of 
copies of $\PP^1$ or del Pezzo variety $V^{2k}$. (See Section 3.) 
For example, the left-hand side (resp. right-hand side) of the graph in Figure \ref{example_sym} 
yields a smooth Fano polytope which corresponds to $V^4$ 
(resp. $\PP^1 \times \PP^1 \times V^2$). 

\begin{figure}[htb!]
\centering
\includegraphics[scale=0.3]{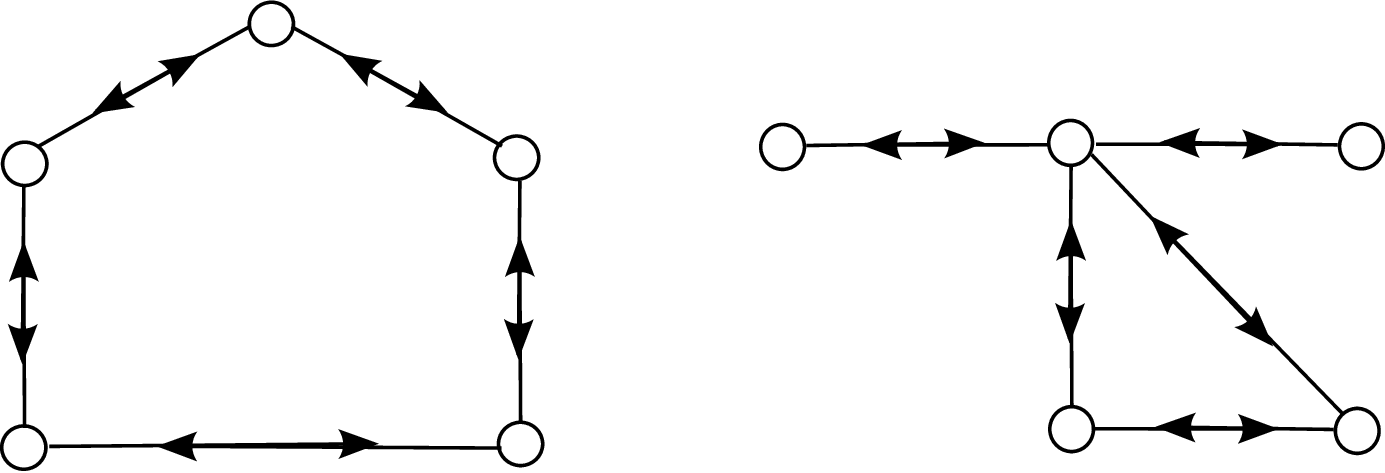}
\caption{directed graphs yielding $V^4$ and $\PP^1 \times \PP^1 \times V^2$}\label{example_sym}
\end{figure}

(b) The left-hand side (resp. right-hand side) of the graph in Figure \ref{example_pseudosym} 
yields a smooth Fano polytope which corresponds to $\widetilde{V}^4$ 
(resp. $\PP^1 \times V^2 \times \widetilde{V}^2$). 

\begin{figure}[htb!]
\centering
\includegraphics[scale=0.3]{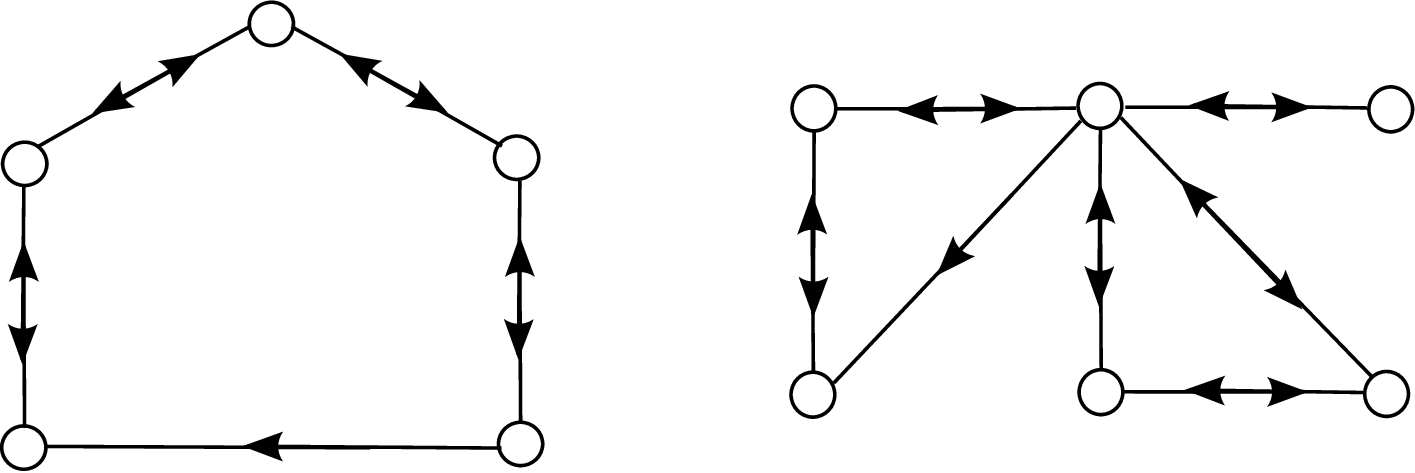}
\caption{directed graphs yielding $\widetilde{V}^4$ and $\PP^1 \times V^2 \times \widetilde{V}^2$}\label{example_pseudosym}
\end{figure}

}\end{Example}

\begin{Example}\label{oldex}{\em 
In \cite{BS}, the definition of a so-called {\em symmetric} 
smooth toric Fano variety is given, which is important from the viewpoint 
whether a smooth toric Fano variety admits a K\"ahler--Einstein metric, 
and some examples of symmetric smooth toric Fano varieties are provided in \cite[Example 4.2 -- 4.4]{BS}. 
(See also \cite{NP}, which gives examples of non-symmetric smooth toric Fano varieties 
admitting K\"ahler--Einstein metric.) 
Note that smooth toric Fano varieties corresponding to centrally symmetric smooth Fano polytopes 
and direct products of copies of projective spaces are symmetric.

Let $m$ be a positive integer and $G_1$ a directed graph 
with its arrow set 
$$A(G_1)=\{(1,2),(2,3),\ldots,(2m+1,2m+2),(2m+2,1),(1,m+2),(m+2,1)\}.$$ 
Then $\Pc_{G_1}$ is a smooth Fano polytope of dimension $2m+1$ 
which corresponds to the example of the case with $k=1$ described in \cite[Example 4.2]{BS}. 

Let $G_2$ be a directed graph with its arrow set 
$$A(G_2)=A(G_1) \cup \{(1,2m+3),(2m+3,1),(m+2,2m+3),(2m+3,m+2)\}.$$ 
Then $\Pc_{G_2}$ is a smooth Fano polytope of dimension $2m+2$ 
which is the example of the case with $k=1$ described in \cite[Example 4.3]{BS}. 
}\end{Example}

\begin{Example}\label{newex}{\em 
By generalizing the above graphs $G_1$ and $G_2$, we obtain a new family of 
symmetric smooth toric Fano varieties. 
For a positive integer $m$ and nonnegative integers $p,q$ with $p \geq q$, let 
$G_{m,p,q}$ denote the directed graph on the vertex set $\{1,\ldots,2m+p+q\}$ with the arrow set 
\begin{align*}
&A(G_{m,p,q})=\{(1,2),(2,3),\ldots,(2m+1,2m+2),(2m+2,1)\} \; \cup \\
&\quad\quad\quad\quad
\{(i_k,i_{k+1}), (i_{k+1}, i_k) : 1 \leq k \leq p\} \; \cup \;
\{(j_\ell, j_{\ell+1}), (j_{\ell+1}, j_\ell) : 1 \leq \ell \leq q \}, 
\end{align*}
where 
\begin{align*}
i_k=
\begin{cases}
1 &\text{if }k=1, \\
2m+1+k &\text{if }k=2,\ldots,p, \\
m+2 &\text{if }k=p+1 
\end{cases}
\; \text{ and }\; j_\ell=
\begin{cases}
1 &\text{if }\ell=1, \\
2m+p+\ell &\text{if }\ell=2,\ldots,q, \\
m+2 &\text{if }\ell=q+1. 
\end{cases}
\end{align*}
Notice that $G_{m,1,0}=G_1$ and $G_{m,2,1}=G_2$. 
It then follows from Theorem \ref{main} that $\Pc_{G_{m,p,q}}$ is a smooth Fano polytope of dimension $2m+p+q-1$ 
if and only if the integers $m,p,q$ satisfy one of the following conditions: 
\begin{align}\label{jouken}
p+q \text{ is odd and } m \geq q > 0 \;\;\text{ or }\;\; m \geq p \text{ and }q=0. 
\end{align}
Here, it is easy to see that $\Pc_{G_{m,p,q}}$ is unimodularly equivalent to the convex hull of 
\begin{align*}
&\eb_1, \eb_2, \ldots,\eb_{2m}, \\
&-(\eb_1+\eb_2+\cdots+\eb_m+\eb_{2m+1}), \; -(\eb_{m+1}+\eb_{m+2}+\cdots+\eb_{2m}-\eb_{2m+1}), \\
&\pm \eb_{2m+2}, \pm \eb_{2m+3}, \ldots, \pm \eb_{2m+p}, \pm (\eb_{2m+1}+\eb_{2m+2}+\eb_{2m+3}+\cdots+\eb_{2m+p}), \\
&\pm \eb_{2m+p+1}, \pm \eb_{2m+p+2}, \ldots, \pm \eb_{2m+p+q-1}, 
\pm (\eb_{2m+1}+\eb_{2m+p+1}+\cdots+\eb_{2m+p+q-1}). 
\end{align*}
Then there exists an automorphism $\sigma_1$ of order 2 defined by 
\begin{align*}
&\sigma_1(\eb_i)=\eb_{i+m}, \; \sigma_1(\eb_{m+i})=\eb_i \;\text{ for }\; 1 \leq i \leq m, \\
&\sigma_1(\eb_j)=-\eb_j \; \text{ for }\; 2m+1 \leq j \leq 2m+p+q-1. 
\end{align*}
There also exists an automorphism $\sigma_2$ of order $m+1$ defined by 
\begin{align*}
&\sigma_2(\eb_i)=\eb_{i+1}, \; \sigma_2(\eb_{m+i})=\eb_{m+i+1} \;\text{ for }\; 1 \leq i \leq m-1, \\
&\sigma_2(\eb_m)=-(\eb_1+\cdots+\eb_m+\eb_{2m+1}), \;\; \sigma_2(\eb_{2m})=-(\eb_{m+1}+\cdots+\eb_{2m}-\eb_{2m+1}), \\
&\sigma_2(\eb_j)=\eb_j \; \text{ for }\; 2m+1 \leq j \leq 2m+p+q-1. 
\end{align*}
Since the common fixed point set of $\sigma_1$ and $\sigma_2$ is only the origin, 
the smooth toric Fano varieties corresponding to $\Pc_{G_{m,p,q}}$, where $m,p,q$ satisfy \eqref{jouken}, 
are symmetric by \cite[Proposition 3.1]{BS}. 
Thus, those admit K\"ahler--Einstein metrics by \cite[Theorem 1.1]{BS}. 
}\end{Example}

\end{document}